\documentclass[a4paper]{amsart}
\usepackage{hyperref}
\usepackage{cleveref}

\usepackage{caption}
\usepackage{subcaption}
\usepackage{pgfplots}
\usepgfplotslibrary{groupplots}
\usepackage{amsmath,amssymb,amsfonts, amsthm}
\usepackage[foot]{amsaddr}
\usepackage{nicefrac}
\usepackage{bm}
\usepackage{enumitem}

\newcommand{\R}{\mathbb{R}}

\newcommand{\N}{\mathbb{N}}
\newcommand{\eps}{\varepsilon}
\newcommand{\Ph}{\R}

\newcommand{\diff}[1]{\partial_{#1}}
\newcommand{\D}{\R^d}
\newcommand{\Ddk}{\left(\R^d\right)^k}

\newcommand{\TD}{\R_+}
\newcommand{\DT}{\D\times\TD}

\newcommand{\sigmaalgebra}{{\mathcal{X}}}

\newcommand{\dd}{\ensuremath{\;d}}
\newcommand{\Dx}{{\ensuremath{\Delta}}}

\newcommand{\Dt}{{\ensuremath{\Delta t}}}
\newcommand{\Ypair}[2]{\bigl\langle #1, #2\bigr\rangle}
\newcommand{\NumericalEvolution}[1]{\ensuremath{\mathcal{S}^{#1}}}
\newcommand{\Evolution}{\ensuremath{\mathcal{S}}}

\newcommand{\cm}{\ensuremath{\bm{\nu}}}

\newcommand{\Prb}{\ensuremath{\mathbb{P}}}
\newcommand{\Prbm}{\ensuremath{\mathcal{P}}}

\newcommand{\FKMT}[3]{\ensuremath{E^{#1}_{#3}(#2)}}
\newcommand{\divx}[1]{\ensuremath{\nabla_x \cdot #1}}
\newcommand{\gradx}[1]{\ensuremath{\nabla_x #1}}

\renewcommand{\vec}[1]{\ensuremath{\mathbf{#1}}}

\newcommand{\Wasserstein}[1]{\ensuremath{W_{#1}}} 
\newcommand{\WorkFVM}[2]{\text{Work}_{\text{FVM}}(#1, #2)}
\newcommand{\WorkFVMCFL}[1]{\text{Work}_{\text{FVM}}^{#1}}
\newcommand{\WorkMLMC}[1]{\text{Work}_{\text{MLMC}}^{#1}}
\newcommand{\WorkMC}[2]{\text{Work}_{\text{MC}}(#1,#2)}

\newcommand{\bigO}{\mathcal{O}}

\newcommand{\ddt}{\ensuremath{\frac{\text{d}}{\text{d}t}}}
\newcommand{\hf}{{\ensuremath{\nicefrac{1}{2}}}}
\newcommand{\iphf}{{i+\hf}}
\newcommand{\imhf}{{i-\hf}}
\usepackage{subcaption}
\newcommand{\pushforward}[2]{#1\##2}
\renewcommand{\leq}{\leqslant}
\renewcommand{\geq}{\geqslant}

\DeclareMathOperator{\Var}{Var}

\DeclareMathOperator{\TV}{TV}

\def\Xint#1{\mathchoice
{\XXint\displaystyle\textstyle{#1}} 
{\XXint\textstyle\scriptstyle{#1}} 
{\XXint\scriptstyle\scriptscriptstyle{#1}} 
{\XXint\scriptscriptstyle\scriptscriptstyle{#1}} 
\!\int}
\def\XXint#1#2#3{{\setbox0=\hbox{$#1{#2#3}{\int}$ }
\vcenter{\hbox{$#2#3$ }}\kern-.56\wd0}}
\def\intavg{\Xint-}

\usepackage{todonotes}

\newtheorem{theorem}{Theorem}
\theoremstyle{definition}
\newtheorem{Definition}{Definition}
\newtheorem{corollary}{Corollary}
\newtheorem{remark}{Remark}

\newlength\figureheight
\newlength\figurewidth

\usepackage{booktabs}
\usepackage{algorithmicx}
\usepackage{algpseudocode}
\usepackage{graphicx}

\usepackage{etoolbox}

\newtoggle{usetikz}
\togglefalse{usetikz}
\newcommand{\InputImage}[3]{}
\iftoggle{usetikz}{%
	\renewcommand{\InputImage}[3]{%
		\setlength\figureheight{#2}%
		\setlength\figurewidth{#1}%
		\input{img_tikz/#3.tikz}%
	}%
} { %
	\renewcommand{\InputImage}[3]{%
		
		\includegraphics[width=#1]{#3.pdf}%
	}%
}
\pgfplotsset{every tick label/.append style={font=\tiny}}
\pgfplotsset{
	tick label style = {font = {\fontsize{6 pt}{12 pt}\selectfont}},
	label style = {font = {\fontsize{6 pt}{12 pt}\selectfont}},
	legend style = {font = {\fontsize{6 pt}{12 pt}\selectfont}},  
}

\usepgfplotslibrary{external} 
\title[Computing statistical solutions]{Numerical approximation of statistical solutions \\ of scalar conservation laws}

\author{U. S. Fjordholm}
\address{Department of Mathematics, University of Oslo, Postboks 1053 Blindern, 0316 Oslo, Norway}
\email{ulriksf@math.uio.no}
\author{K. Lye \and S. Mishra}
\address{Seminar for Applied Mathematics, ETH Z\"urich, R\"amistrasse 101, 8092 Z\"urich, Switzerland}
\email[K. Lye]{kjetil.lye@sam.math.ethz.ch}
\email[S. Mishra]{siddhartha.mishra@sam.math.ethz.ch}

\begin{document}
\begin{abstract}
We propose efficient numerical algorithms for approximating statistical solutions of scalar conservation laws. The proposed algorithms combine finite volume spatio-temporal approximations with Monte Carlo and multi-level Monte Carlo discretizations of the probability space. Both sets of methods are proved to converge to the entropy statistical solution. We also prove that there is a considerable gain in efficiency resulting from the multi-level Monte Carlo method over the standard Monte Carlo method. Numerical experiments illustrating the ability of both methods to accurately compute multi-point statistical quantities of interest are also presented.
\end{abstract}

\maketitle


\section{Introduction}

Hyperbolic systems of conservation laws are a large class of nonlinear partial differential equations that arise in a wide variety of models in physics and engineering \cite{Dafermos}. Prominent examples include the shallow water equations of oceanography, compressible Euler equations of aerodynamics and the magnetohydrodynamics (MHD) equations of plasma physics. 

The simplest example for this class of PDEs is provided by the scalar conservation law
\begin{equation}
	\label{eq:conservation_law}
	\begin{split}
	u_t+\divx{f(u)} = 0 \\
	u(x,0) = \bar{u}(x)
	\end{split}
\end{equation}
where $u : \DT\to\R$ is the conserved variable and $f:\R\to\R^d$ is the flux function.  A prototypical example of a scalar conservation law is the Burgers equation, i.e \eqref{eq:conservation_law} with $d=1$ and flux function $f(u) = \frac{1}{2} u^2$.  

It is well known that discontinuities, such as shock waves, can arise in the solution of \eqref{eq:conservation_law} even when the initial data $\bar{u}$ is smooth. Hence, one seeks solutions of \eqref{eq:conservation_law} in the sense of distributions. These \emph{weak} solutions are not necessarily unique and need to be augmented with additional admissibility criteria called \emph{entropy conditions} to ensure uniqueness. \emph{Entropy solutions} of the multi-dimensional scalar conservation law \eqref{eq:conservation_law} exist as long as $\bar{u}\in L^1\cap L^\infty(\R^d)$, and are unique and stable with respect to the initial data:
\[
\big\|\Evolution(t)\bar{u}-\Evolution(t)\bar{v}\big\|_{L^1(\R^d)} \leq \|\bar{u}-\bar{v}\|_{L^1(\R^d)} \qquad \forall\ \bar{u},\bar{v}\in L^1\cap L^\infty(\R^d).
\]
Here, $\Evolution(t) : \bar{u} \mapsto u(t)$ denotes the data to solution operator of \eqref{eq:conservation_law}.

On the other hand, similar well-posedness results for systems of conservation laws are only available in one space dimension \cite{Bressan}. There are no global existence results for generic multi-dimensional systems of conservation laws. Moreover,  it is now known that entropy solutions for (some examples) of multi-dimensional systems may not be unique \cite{delellis_nonunique1,delellis_isentropic}.  

Given this paucity of well-posedness results for systems of conservation laws, it is natural to seek alternative solution frameworks. A possible solution paradigm is that of entropy measure valued solutions \cite{diperna}, where the sought for solution is not necessarily an integrable function but rather a  Young measure, i.e.~a space-time parametrized probability measure. Although entropy measure-valued solutions exist globally (even for multi-dimensional systems) and can be approximated with Monte Carlo type ensemble averaging numerical algorithms \cite{fkmt}, these solutions are not unique, even for scalar conservation laws \eqref{eq:conservation_law}, in particular in the case when the initial Young measure is not concentrated on a single integrable function. The reason behind this non-uniqueness is the lack of information on spatial correlations (conditional probabilities) that is intrinsic to the notion of Young measures.  

In a recent paper \cite{FLM17}, the authors proposed another solution concept, namely \emph{statistical solutions}, to supplement measure-valued solutions with information about multi-point spatial correlations. Statistical solutions are time-parametrized probability measures on $L^p(D)$, where $D \subset \R^d$ is the spatial domain and $1\leq p < \infty$. It was proved in \cite{FLM17} that these probability measures on spaces of $p$-integrable functions are identified with a hierarchy of Young measures, termed as \emph{correlation measures}. Consequently, the time-evolution of statistical solutions is completely determined in terms of an infinite family of nonlinear PDEs for a class of moments of the corresponding correlation measures. It was shown in \cite{FLM17} that \emph{entropy statistical solutions} of the scalar conservation law \eqref{eq:conservation_law} always exist, are unique and stable with respect to initial data. This is in contrast to earlier notions of statistical solutions for Burgers' equation that are defined as probability measures on distributions \cite{Duchon1, Duchon2,Bertoin}. See also \cite{FoiasTemam} for a different notion of statistical solutions for the incompressible Navier--Stokes equations. 

Statistical solutions can also be considered as a framework for uncertainty quantification (UQ), i.e.~the modeling, analysis and efficient numerical approximation of uncertainty in the solutions of PDEs, given uncertainties in their inputs such as the initial and boundary data, fluxes, coefficients etc.~\cite{UQbook}. For instance, the initial datum $\bar{u}$ in \eqref{eq:conservation_law} is a measured (observed) quantity, and these measurements are not exact. Often, one needs to model the underlying input uncertainty and the resulting solution uncertainty statistically. It is common practice to model both input and solution uncertainty in terms of random fields (often characterized by a large but finite number of parameters). A popular notion of uncertainty modeling for conservation laws is that of \emph{random entropy solutions} introduced in \cite{mlmc_hyperbolic,mss1}. However,  random entropy solutions rely on a particular parametrization of the underlying uncertainty. On the other hand, the framework of statistical solutions offers a parametrization-independent modeling of uncertainty where only the  law of random fields, i.e.~a probability measure on the underlying function space, is described and evolved, cf.~\cite{FLM17,abgrall_mishra}.  

The study of scalar conservation laws with random initial conditions also plays an important role in turbulence research \cite{fris1}. The results on \emph{burgulence} can also be reinterpreted within the framework of statistical solutions of \eqref{eq:conservation_law}.

Given the above discussion, we focus on the efficient numerical approximation of statistical solutions for scalar conservation laws \eqref{eq:conservation_law} in this article. We will approximate the spatio-temporal domain with standard finite volume/finite difference numerical schemes. The probability space is {discretized} with a Monte Carlo method. We prove that that this finite-volume Monte Carlo method converges, in an appropriate topology, to the entropy statistical solution of \eqref{eq:conservation_law} as the mesh is refined and the number of Monte Carlo samples is increased. A complexity analysis demonstrates that this procedure can be rather expensive on account of the slow convergence (in terms of samples) of the Monte Carlo method. 

In order to improve computational efficiency, we propose and analyze a multi-level Monte Carlo (MLMC) version of our numerical method. MLMC methods were first proposed in \cite{Heinrich2001,giles} and were applied to UQ for conservation laws in \cite{mlmc_hyperbolic}. We adapt the MLMC procedure to our setup and prove that the MLMC method also converges to the entropy statistical solution. Moreover, we show that the MLMC method is significantly more efficient than its Monte Carlo counterpart. We illustrate the performance of both set of methods on a suite of numerical test problems.

The rest of the article is organized as follows. We define statistical solutions of \eqref{eq:conservation_law} in \Cref{sec:stat}. Finite volume schemes for approximating the underlying deterministic problem are described in \Cref{sec:fv}. Monte Carlo and multi-level Monte Carlo methods are presented in \Cref{sec:mc,sec:mlmc}, respectively. We present a set of numerical experiments illustrating our methods in \Cref{sec:numex}. 


\section{Statistical solutions of conservation laws}
\label{sec:stat}
In this section we briefly describe the framework of statistical solutions introduced in \cite{FLM17}. As mentioned above, we are interested in the situation where instead of an initial data $\bar{u}\in L^1(\R^d)$ for \eqref{eq:conservation_law}, we are given some $\bar{\mu}\in \Prbm\big(L^1(\R^d)\big)$, that is, a probability distribution over different initial data $\bar{u}\in L^1(\R^d)$. A statistical solution of this initial value problem is a map $t \mapsto \mu_t\in\Prbm\big(L^1(\R^d)\big)$ which satisfies the PDE \eqref{eq:conservation_law} in a certain sense. In \cite{FLM17}, the authors showed that any probability measure $\mu\in\Prbm(L^1(\R^d))$ can be described equivalently as a \emph{correlation measure}---a hierarchy $\cm=(\nu^1,\nu^2,\dots)$ in which each element $\nu^k$ provides the joint probability distribution $\nu^k_{x_1,\dots,x_k}$ of the solution values $u(x_1),\dots,u(x_k)$ at any choice of spatial points $x_1,\dots,x_k\in \R^d$. The evolution equation for $\mu_t$ is most naturally described in terms of its corresponding correlation measure, yielding an infinite hierarchy of evolution equations. In particular, the equation for the one-point distribution $\nu^1_x$ coincides with that of \emph{measure-valued solutions}, cf.~DiPerna~\cite{diperna}. Most statistical observables of a fluid flow can be directly expressed in terms of correlation marginals.  For instance, all one-point statistics such as the mean flow can be expressed in terms of $\nu^1$, while the \emph{structure functions} of turbulence theory are easily expressed in terms of $\nu^2$ (see Section \ref{sec:numex}).

\newcommand{\phase}{\mathcal{U}}
\begin{Definition}
Let $d,N\in\N$, let $q\in[1,\infty)$, let $D\subset\R^d$ be an open set (the ``spatial domain'') and (for notational convenience) denote $\phase = \R^N$ (``phase space''). A \textit{correlation measure from $D$ to $\phase$} is a collection $\cm = (\nu^1,\nu^2,\dots)$ of maps satisfying for every $k\in\N$:
\begin{enumerate}[label=\it(\roman*)]
\item $\nu^k$ is a Young measure from $D^k$ to $\phase^k$.
\item \textit{Symmetry:} if $\sigma$ is a permutation of $\{1,\dots,k\}$ and $f\in C_0(\phase^k)$ then $\Ypair{\nu^k_{\sigma(x)}}{f(\sigma(\xi))} = \Ypair{\nu^k_{x}}{f(\xi)}$ for a.e.\ $x\in D^k$. 
\item \textit{Consistency:} If $f\in C_b(\phase^k)$ is of the form $f(\xi_1,\dots,\xi_k) = g(\xi_1,\dots,\xi_{k-1})$ for some $g\in C_0(\phase^{k-1})$, then $\Ypair{\nu^k_{x_1,\dots,x_k}}{f} = \Ypair{\nu^{k-1}_{x_1,\dots,x_{k-1}}}{g}$ for almost every $(x_1,\dots,x_k)\in D^k$.
\item \textit{$L^q$ integrability:} 
\begin{equation}\label{eq:corrlqbound_stationary}
\int_D \Ypair{\nu^1_{x}}{|\xi|^q}\,dx < \infty.
\end{equation}
\item\textit{Diagonal continuity (DC):} $\lim_{\eps\to0}d_\eps^q(\nu^2) = 0$, where
\begin{equation}\label{eq:defd_eps}
d_\eps^q(\nu^2) := \left(\int_D \intavg_{B_\eps(x)} \Ypair{\nu^2_{x,y}}{|\xi_1-\xi_2|^q}\,dydx\right)^{1/q}.
\end{equation}
(Here, $\intavg_B = \frac{1}{|B|}\int_B$, the average over $B$.)
\end{enumerate}
\end{Definition}

The sense in which correlation measures and probability measures on $L^p$ are equivalent is made more precise in the following definition:
\begin{Definition}
A probability measure $\mu \in \Prbm(L^q(D;\phase))$ is said to be \emph{dual} to a correlation measure $\cm$ from $D$ to $\phase$ provided
\begin{equation}
\int_{D^k} \Ypair{\nu^k_{x}}{g(x,\cdot)}\, dx = \int_{L^2}\int_{D^k} g(x,v(x_1),\dots,v(x_k))\,dxd\mu(v)
\end{equation}
for every $k\in\N$ and for every Caratheodory function $g : D^k \to C_b(\phase^k)$.
\end{Definition}

It was shown in \cite{FLM17} that every probability measure $\mu \in \Prbm(L^q(D;\phase))$ is dual to a unique correlation measure $\cm$, and \textit{vice versa}. Using this duality, we can now state the definition of statistical solutions.

\begin{Definition}
	We say that a weak*-measurable map $t\mapsto \mu_t\in\Prbm(L^1(\D))$ with corresponding spatial correlation measures $\cm_t = (\nu_t^{k})_{k\in\N}$, is a \emph{statistical solution} of \eqref{eq:conservation_law} with initial data $\bar{\mu}\in\Prbm(L^1(\D))$, if
	\begin{equation}
		\begin{split}
		\int_{\TD}\int_{\Ddk}\Ypair{\nu_{t,x}^{k}}{\xi^1\cdots\xi^k}\diff{t}\phi+	\sum_{i=1}^k\Ypair{\nu_{t,x}^{k}}{\xi^1\cdots f(\xi^i)\cdots \xi^k}\gradx\phi\dd x \dd t \\ +\int_{\Ddk}\Ypair{\bar{\nu}^k_x}{\xi^1\cdots\xi^k}\phi|_{t=0}\dd x=0,
		\end{split}
	\end{equation}
	for all $\phi\in C_c^\infty\big(\Ddk\times\TD\big)$ and all $k\in\N$.
\end{Definition}


We define the \emph{canonical statistical solution} as $\mu_t = \pushforward{\Evolution(t)}{\bar{\mu}}$, that is, the measure satisfying
\[
\int_{L^1} G(u)\,d\mu_t(u) = \int_{L^1} G(\Evolution(t)\bar{u})\,d\bar{\mu}(\bar{u})
\]
for every $G\in C_b(L^1(\D))$. It is straightforward to show that this indeed constitutes a statistical solution of \eqref{eq:conservation_law} with initial data $\bar{\mu}$. Moreover, in \cite{FLM17} it was shown that under an additional entropy condition on the statistical solution, any two \emph{entropy statistical solutions} (see Section 4.2 of \cite{FLM17} for a precise definition) $\mu_t,\rho_t$ satisfy
\[
\Wasserstein{1}(\mu_t,\rho_t) \leq \Wasserstein{1}(\bar{\mu},\bar{\rho}),
\]
where $\Wasserstein{1}$ is the 1-Wasserstein distance on the set of probability measures $\Prbm(L^1(\D))$. In particular, the entropy statistical solution is unique and coincides with the canonical statistical solution. We refer to \cite{FLM17} for further details.

\section{Numerical approximation of statistical solutions}
\label{sec:fv}
In this section we introduce numerical approximations of statistical solutions, which are based on standard finite volume methods (FVM) for \eqref{eq:conservation_law}. In Section \ref{sec:fvm} we give a short description of FVM for scalar conservation laws, and in Section \ref{sec:convstatsoln} we prove that FVM applied to statistical ensembles converges as the mesh is refined. We postpone the discretization of the probability space to Section \ref{sec:montecarlo}.

\subsection{Finite volume methods for conservation laws}\label{sec:fvm}
This section briefly describes the conventional approach of numerically approximating conservation laws through finite volume and finite difference methods. For a complete review, one can consult~\cite{leveque_green}.

We discretize the computational spatial domain as a collection of cells
\[\big\{(x^1_{i_1-\hf}, x^1_{i_1+\hf})\times\cdots\times(x^d_{i^d-\hf}, x^d_{i^d+\hf})\big\}_{(i^1,\ldots,i^d)}\subset \D,\]
with corresponding cell midpoints
\[x_{i^1,\ldots, i^d}:=\left(\frac{x^1_{i_1+\hf}+x^1_{i_1-\hf}}{2},\ldots,\frac{x^d_{i_d+\hf}+x^d_{i_d-\hf}}{2}\right).\]
For simplicity, we assume that our mesh is equidistant, that is,
\[x^k_{i^k+\hf} - x^k_{i^k-\hf} \equiv \Dx \qquad \forall\ k=1,\dots,d\]
for some $\Dx>0$. We will describe the semi-discrete case. For each cell, marked by $(i^1,\ldots, i^d)$, we let $u^{\Dx}_{i^1,\ldots, i^d}(t)$ denote the averaged value in the cell at time $t\geq 0$.

We use the following semi-discrete formulation:
\begin{equation}\label{eq:semi_d}
\begin{split}
\ddt{}u^{\Dx}_{i^1,\ldots,i^d}(t)+\sum_{k=1}^d\frac{1}{\Dx}\left(F^{k,\Dx}_{i^1,\ldots,i^k+\hf,\ldots, i^d}(t)-F^{k,\Dx}_{i^1,\ldots,i^k-\hf,\ldots, i^d}(t)\right)= 0  \\
u^{\Dx}_{i^1,\ldots,i^d}(0) = u_0(x_{i^1,\ldots,i^d})
\end{split}
\end{equation}
where $F^{k,\Dx}$ is a \emph{numerical flux function} in direction $k$. In a \emph{$(2p+1)$-point scheme,} the numerical flux function $F^{k,\Dx}_{i^1,\ldots,i^k+hf,\ldots,i^d}(t)$ can be written as a function of $u^{\Dx}_{i^1,\ldots,j^k,\ldots,i^k}(t)$ for $j^k=i^k-p+1,\ldots,i^k+p$. We furthermore assume the numerical flux function is consistent with $f$ and locally Lipschitz continuous, which amounts to requiring that for every bounded set $K\subset \R$, there exists a constant $C>0$ such that for $k=1,\ldots,d$,
\begin{equation}
	|F^{k,\Dx}_{i^1,\ldots,i^d}(t)-f(u^{\Dx}_{i^1,\ldots,i^d})|\leq C\sum_{j^k=i^1-p+1}^{i^k+p}|u^{\Dx}_{i^1,\ldots,i^d}(t)-u^{\Dx}_{i^1,\ldots, j^k,\ldots,i^d}(t)|,
\end{equation}
whenever $u^{\Dx}_{i^1,\ldots, i^k-p+1,\ldots,i^d}(t),\ldots, u^{\Dx}_{i^1,\ldots, i^k+p,\ldots, i^d}(t) \in K$.

We let
$\NumericalEvolution{\Dx}:L^\infty(\D)\to L^\infty(\DT)$
be the spatially discrete numerical evolution operator corresponding to \eqref{eq:semi_d}. Since $\NumericalEvolution{\Dx}$ is the composition of a projection from $L^\infty$ onto piecewise constant functions and a continuous evolution under an ordinary differential equation, we see that $\NumericalEvolution{\Dx}$ is measurable.

The current form of \eqref{eq:semi_d} is continuous in time, and one needs to employ a time stepping method to discretize the ODE in time, usually through some strong stability preserving Runge--Kutta method \cite{GST}.

\subsection{Spatial convergence of statistical solutions for scalar equations}\label{sec:convstatsoln}
We are interested in measuring convergence to the statistical solution of \eqref{eq:conservation_law} in the $1$-Wasserstein distance on $\Prbm(L^1(\D))$, defined here as
\begin{equation}
	\label{eq:wasserstein}
	\Wasserstein{1}(\mu_1, \mu_2)=\inf_{\pi\in\Pi(\mu_1,\mu_2)}\left(\int_{L^1(\D)^2}\|u_1-u_2\|_{L^1(\D)}\dd \pi(u_1,u_2)\right).
\end{equation}
Let $\NumericalEvolution{\Dx}$ be some numerical evolution operator, then we set
\begin{equation}
	\mu^\Dx_t:=\pushforward{\NumericalEvolution{\Dx}(t)}{\bar{\mu}}.
\end{equation}

\begin{theorem}
\label{thm:fv}
	Let $\bar{\mu}\in\Prbm(L^1(\D))$ be the initial data.
    For every $t\in\R^+$, define $\mu^\Dx_t:=\pushforward{\NumericalEvolution{\Dx}(t)}{\bar{\mu}}$ and $	\mu_t:=\pushforward{\Evolution(t)}{\bar{\mu}}$. Assume the numerical scheme has a convergence rate of $s$ for the given initial data, i.e,
    \begin{equation} 
    \label{eq:convergence_requirement}
    \int_{L^1(\D)}\big\|\Evolution(t)(\bar{u})- \NumericalEvolution{\Dx}(t)(\bar{u})\big\|_{L^1(\D)}\dd \bar{\mu}(\bar{u})\leq C\Dx^s, \end{equation}
    for some $s>0$. Then 
	\begin{equation}
		\label{eq:spatial_wasserstein_scalar}
		\Wasserstein{1}(\mu^\Dx_t, \mu_t)\leq C\Dx^{s}.
	\end{equation}
\end{theorem}

\begin{proof}
	Define the measure $\gamma\in\Prbm\big(L^1(\D)\times L^1(\R^d)\big)$ by
	\begin{equation}
		\gamma = \pushforward{\big(\Evolution(t), \NumericalEvolution{\Dx}(t)\big)}{\bar{\mu}}.
	\end{equation}
	One can trivially check that  $\gamma\in \Pi(\mu_t,\mu_t^{\Dx})$. We obtain
	\begin{align*}
		\Wasserstein{1}(\mu_t, \mu^\Dx_t) & \leq \int_{L^1(\D)^2}\|u_1-u_2\|_{L^1(\D)}\dd \gamma(u_1,u_2)  \\
		& = \int_{L^1(\D)}\big\|\Evolution(t)(\bar{u})- \NumericalEvolution{\Dx}(t)(\bar{u})\big\|_{L^1(\D)}\dd \bar{\mu}(\bar{u}). \\
	\end{align*}
	Inserting the assumptions of \eqref{eq:convergence_requirement} yields the claim. 
\end{proof}
The above theorem is used below to obtain convergence rates for monotone schemes approximating scalar conservation laws (see \cite{GR,HR} for definitions).
\begin{corollary}
	Let $\bar{\mu}\in\Prbm(L^1(\D))$ be such that
	\[
	\TV(\bar{\mu}):=\int_{L^1(\D)} \TV(\bar{u})\dd \bar{\mu}(\bar{u})<\infty.
	\]
	For every $t\in\R^+$, define $\mu_t:=\pushforward{\Evolution(t)}{\bar{\mu}}$ and $\mu^\Dx_t:=\pushforward{\NumericalEvolution{\Dx}(t)}{\bar{\mu}}$. Then
	\begin{equation}
	\label{eq:spatial_wasserstein_scalar_kuznetsov}
	\Wasserstein{1}(\mu^\Dx_t, \mu_t)\leq C\Dx^{1/2}\TV(\bar{\mu}).
	\end{equation}
\end{corollary}
\begin{proof}
	The convergence estimate of Kuznetsov \cite{KUZNETSOV1976105} for monotone schemes approximating scalar conservation laws yields
	\[\big\|\Evolution(t)(\bar{u})- \NumericalEvolution{\Dx}(t)(\bar{u})\big\|_{L^1(\D)}\leq C\sqrt{\Dx}\TV(\bar{u})\]
	for every $\bar{u}\in L^1$ with $\TV(\bar{u}) < \infty$. The estimate \eqref{eq:spatial_wasserstein_scalar_kuznetsov} now follows from \eqref{eq:spatial_wasserstein_scalar}.
\end{proof}



\section{Monte Carlo approximations of statistical solutions}\label{sec:montecarlo}
\label{sec:mc}
Let $\bar{\mu}\in \Prbm(L^1(\D))$ be the given initial measure. We choose a probability space $(\Omega,\sigmaalgebra,\Prb)$ and a random field $\bar{u}\in L^2\big(\Omega; L^1(\D)\big)$ such that the law of $\bar{u}$ with respect to $\Prb$ is $\bar{\mu}$. In most real-world applications, the uncertainty in initial data is usually described in terms of such a random field $\bar{u}$, for instance, one given as a parametric function $\bar{u}:[0,1]^Q\times \D\to \Ph$, with possibly $Q \gg 1$. Draw $M$ independent and identically distributed (i.i.d.) samples $\bar{u}_1,\dots,\bar{u}_M$ of $\bar{u}$ and for $1 \leq k \leq M$, set \[u_k(\omega, \cdot, t):=\NumericalEvolution{\Dx}(t)(\bar{u}_k).\]
We define the Monte Carlo $\mu_t^{\Dx,M}$ approximation as
\[\mu_t^{\Dx,M}:=\frac{1}{M}\sum_{k=1}^M \delta_{u_k(\omega; \cdot, t)},\]
or equivalently, through its action on an integrable function,
\[\Ypair{\mu^{\Dx,M}_t}{G}=\frac{1}{M}\sum_{k=1}^M G(u_k(\omega; \cdot, t))\qquad G\in C_b(L^1(\D)).\]

The relevant notion of convergence on $\Prbm(L^1(\D))$ is that of weak convergence. We say that a sequence $\{\nu_n\}_{n\in\N}\subset \Prbm(L^1(\D))$ converges weakly to $\nu\in \Prbm(L^1(\D))$ if
\[\Ypair{\nu_n}{G}\to\Ypair{\nu}{G}\qquad \text{for all } G\in C_b(L^1(\D)).\]
It can be shown that weak convergence preserves the probability measure structure, that is, $\nu \geq 0$ and $\nu(X)=1$. 

Using standard Monte Carlo techniques we obtain the following convergence theorem for our Monte Carlo method.

\begin{theorem}\label{thm:mc}
	Let $\bar{\mu}\in\Prbm\big(L^1(\D)\big)$ be the initial data.
	Furthermore, let $t\in\R_+$, $\mu^\Dx_t:=\pushforward{\NumericalEvolution{\Dx}(t)}{\bar{\mu}}$ and let  $\mu^{\Dx,M}_t$ be the Monte Carlo approximation. Then
	\begin{equation}
		\label{eq:weak_conv_mc}
		\big\|\Ypair{\mu^{\Dx,M}_t-\mu^\Dx_t}{G}\big\|_{L^2(\Omega)}= \frac{\sqrt{\Var(G(u^{\Delta x}(\cdot, \cdot, t)))}}{M^{1/2}}
	\end{equation}
	for all $G\in C\big(L^1(\D)\big)$. Here,
	\begin{equation}
		\label{eq:mc_assumption}
		\Var(G(u^{\Delta x}(\cdot, \cdot, t))) := \int_{L^1(\D)}G(u)^2\dd \mu_t(u)-\bigg(\int_{L^1(\D)}G(u)\dd \mu_t(u)\bigg)^2.
	\end{equation}
\end{theorem}
\begin{proof}
	We have
	\begin{equation}
		\Ypair{\mu^\Dx_t}{G} = \int_{\Omega} G(u^{\Dx}(\omega', \cdot, t))\dd \Prb(\omega').
	\end{equation}
	Hence,
	\begin{align*}
    \big\|\Ypair{\mu^{\Dx,M}_t-\mu^\Dx_t}{G}\big\|_{L^2(\Omega)}^2 & =\int_\Omega\bigg[\int_{\Omega} G(u^{\Dx}(\omega', \cdot, t))\dd \Prb(\omega')-\frac{1}{M}\sum_{k=1}^M G(u_k(\omega; \cdot, t))\bigg]^2\dd\Prb(\omega) \\
	& = \frac{\Var(G(u^{\Delta x}(\cdot, \cdot, t)))}{M},
	\end{align*}
	where the key observation is that $G(u^{\Dx}(\omega, \cdot, t))$ is a real-valued random variable, hence we can directly apply the standard Monte Carlo estimate found in \cite{ledoux1991}. 
\end{proof}
Note that if $G\in C_b(L^1(\D))$ then $\Var(G(u^{\Delta x}(\cdot, \cdot, t)))\leq 2\|G\|_{C_b(L^1)}^2<\infty$, and hence the right-hand side of \eqref{eq:weak_conv_mc} is finite. Examples of such test functions $G$ include finite-dimensional observables $G(u)=g\big(\Ypair{\phi_1}{u},\dots,\Ypair{\phi_n}{u}\big)$ for functionals $\phi_1,\dots,\phi_n\in L^1(\R^d)^*=L^\infty(\R^d)$ and $g\in C_b(\R^n)$.
\begin{remark}
	It can be shown that the convergence rate of the Monte Carlo method, measured in the Wasserstein metric for probability measures on $\R^d$, deteriorates as $d$ grows \cite{wasserstein_mc_convergence}, and hence one can not expect to obtain a convergence rate for the Monte Carlo method in the Wasserstein metric for probability measures on $L^1$. However, as the weak topology on probability measures is metrized by the Wasserstein metric, Theorem \ref{thm:mc} enables us to conclude that the Monte Carlo method does converge with respect to the Wasserstein metric as $M \rightarrow \infty$.
\end{remark}

\subsection{Work analysis for Monte Carlo}
\label{subsec:work_analysis_mc}
The \emph{work} of a numerical method is the number of floating point operations it consumes. The classical explicit finite volume method has a work estimate of
\begin{equation}
	\WorkFVM{\Dx}{\Dt} = \bigO(\Dx^{-d}\Dt^{-1}).
\end{equation}
Applying the CFL requirement $\Dt \simeq \Dx$, gives
\[\WorkFVM{\Dx}{\Dt} = \WorkFVMCFL{\Dx} = \bigO(\Dx^{-d-1}).\]
Thus, the work to compute $\FKMT{\Dx}{\bar{u}}{M}$ scales as
\[\WorkMC{\Dx}{M}=M\WorkFVMCFL{\Dx}=\bigO(M\Dx^{-d-1}).\]
If we assume the spatial error scales as
\[\Wasserstein{1}(\mu_t,\mu_t^\Dx)=\bigO(\Dx^s),\]
we then choose the number of samples such that the Monte Carlo error is asymptotically the same as the spatial error. That is, we choose
\[M^{-1/2}\simeq \Dx^s \quad \Leftrightarrow \quad M\simeq \Dx^{-2s}.\]
This gives the optimal work estimate
\begin{equation}
	\label{eq:work_mc}
	\WorkMC{\Dx}{M}=\bigO(\Dx^{-d-1-2s}).
\end{equation}

\section{Multilevel Monte Carlo approximation}
\label{sec:mlmc}
We define the multi-level Monte Carlo approximation similar to that found in \cite{giles} and \cite{mlmc_hyperbolic}. We assume we have a family of nested meshes with mesh lengths $(\Dx_l)_{l=0}^L$, where
\begin{equation}\label{eq:meshwidth}
\Dx_l=2^{-l}\Dx_0\qquad \text{for } l=1,\ldots, L.
\end{equation}
Following the notation in the previous section, we 	define the multi-level Monte Carlo approximation of $\mu_t$ to be
\begin{equation}
	\label{eq:mlmc}
	\mu^{\vec{\Dx},\vec{M}}_t:=\mu_t^{\Dx_0,M_0}+\sum_{l=1}^L\left(\mu_t^{\Dx_l,M_l}-\mu_t^{\Dx_{l-1},M_l}\right),
\end{equation}
where $\vec{M} = (M_0,\dots,M_L) \in \N^{L+1}$ and $\vec{\Dx} = (\Dx_0, \dots,\Dx_L)$. In other words, we set
\[
\mu^{\vec{\Dx},\vec{M}}_t:=\frac{1}{M_0}\sum_{k=1}^{M_0}\delta_{u_k^{\Dx_0}(\omega; \cdot, t)}+\sum_{l=1}^L\frac{1}{M_l}\sum_{k=1}^{M_l}\left(\delta_{u_k^{\Dx_l}(\omega;\cdot, t)} - \delta_{u_k^{\Dx_{l-1}}(\omega;\cdot, t)}\right).
\]

\begin{remark}
$\mu^{\vec{\Dx},\vec{M}}_t$ will in general not be a probability measure on $L^1(\D,\Ph)$, but rather a signed measure on $L^1(\D,\Ph)$.
\end{remark}

\begin{theorem}
	\label{thm:mlmc}
	Let $\bar{\mu}\in\Prbm(L^1(\D,\Ph))$ be the initial data, and for any $t\in\TD$ and $\Dx>0$, let $\mu^{\Dx}_t:=\pushforward{\NumericalEvolution{\Dx}(t)}{\bar{\mu}}$ and let $\mu^{\vec{\Dx},\vec{M}}_t$ be the Multilevel Monte Carlo approximation \eqref{eq:mlmc}. Then 
	\[
	\big\|\Ypair{\mu^{\vec{\Dx},\vec{M}}_t-\mu^{\Dx_L}_t}{G}\big\|_{L^2(\Omega)}\leq \frac{\Var(G(u^{\Dx_0}(\cdot, \cdot, t)))}{M_0^{1/2}} + \sum_{l=1}^L\frac{\Var\left(G(u^{\Dx_l}(\omega;\cdot, t))-G(u^{\Dx_{l-1}}(\omega;\cdot, t))\right)^{1/2}}{M_l^{1/2}},
	\]
	for every $G\in C_b(L^1(\D, \Ph))$. Moreover,
	\[\Var\left(G(u^{\Dx_l}(\omega;\cdot, t))-G(u^{\Dx_{l-1}}(\omega;\cdot, t))\right)\to 0 \qquad \text{as } \Delta \to 0.\]
\end{theorem}

\begin{proof}
	We observe that
	\[\mu^{\Dx_L}=\mu^{\Dx_0}+\sum_{l=1}^L\left(\mu^{\Dx_l}-\mu^{\Dx_{l-1}}\right).\]
	Therefore, we get
	\begin{align*}
	\big\|\Ypair{\mu^{\vec{\Dx},\vec{M}}_t-\mu^{\Dx_L}_t}{G}\big\|_{L^2(\Omega)}&\leq \big\|\Ypair{\mu^{\Dx_0,M_0}_t-\mu^{\Dx_0}_t}{G}\big\|_{L^2(\Omega)}\\
	&\quad +\sum_{l=1}^L\Bigl\|\Ypair{\mu^{\Dx_l,M_l}_t-\mu^{\Dx_{l-1},M_l}_t}{G} - \Ypair{\mu^{\Dx_l}_t-\mu^{\Dx_{l-1}}_t}{G}\Bigr\|_{L^2(\Omega)}.
	\end{align*}
	We can easily bound the first term using \Cref{thm:mc} to get
	\[\big\|\Ypair{\mu^{\Dx,M_0}_t-\mu^{\Dx_0}_t}{G}\big\|_{L^2(\Omega)}\leq C_GM_0^{-1/2}.\]
	For the second term, we observe that
	\begin{multline*}
	\Bigl\|\Ypair{\mu^{\Dx_l,M_l}_t-\mu^{\Dx_{l-1},M_l}_t}{G} - \Ypair{\mu^{\Dx_l}_t-\mu^{\Dx_{l-1}}_t}{G}\Bigr\|_{L^2(\Omega)} \\
	=\left\|\frac{1}{M_l}\sum_{k=1}^{M_l}\left(G(u_k^{\Dx_l}(\omega;\cdot, t))-G(u_k^{\Dx_{l-1}}(\omega;\cdot, t))\right)-\int_{\Omega}\big(G(u^{\Dx_l}(\omega';\cdot, t))-G(u^{\Dx_{l-1}}(\omega';\cdot, t))\big)\dd \Prb(\omega')\right\|_{L^2(\Omega)}.
	\end{multline*}
	Using a standard Monte Carlo estimate, we get
	\[
	\Bigl\|\Ypair{\mu^{\Dx_l,M_l}_t-\mu^{\Dx_{l-1},M_l}_t}{G} - \Ypair{\mu^{\Dx_l}_t-\mu^{\Dx_{l-1}}_t}{G}\Bigr\|_{L^2(\Omega)} \leq \frac{\Var\left(G(u^{\Dx_l}(\omega;\cdot, t))-G(u^{\Dx_{l-1}}(\omega;\cdot, t))\right)}{M_l^{1/2}}. \]
	Summing over $l$ and adding the first estimate yields the first statement. The last statement follows from the facts that $G$ is continuous and $(u^{\Dx})_{\Dx>0}$ forms a convergent sequence as $\Dx\to 0$.
\end{proof}
\begin{remark}
	For general $G\in C_b(L^1(\D, \Ph))$, we can only guarantee that

	\begin{equation}
		\label{eq:variance_decay_simple}
		\Var\left(G(u^{\Dx_l}(\omega;\cdot, t))-G(u^{\Dx_{l-1}}(\omega;\cdot, t))\right)\to 0.
	\end{equation}
	However, for practical computations, $G$ is known, and we can get a much better estimate on the decay of \eqref{eq:variance_decay_simple}.
\end{remark}

\subsection{Work analysis of MLMC}
In this section we extend the analysis of \Cref{subsec:work_analysis_mc} to the MLMC algorithm. In the computation of the MLMC approximation $\mu^{\vec{\Dx},\vec{M}}$, we compute $M_l$ finite volume simulations with resolution $\Dx_l$ for each $l=0,\dots,L$, and $M_l$ finite volume simulations with resolution $\Dx_{l-1}$ for $l=1,\dots,L$. If $M_l\leq M_{l-1}$ then the latter can be neglected, and we find that the work performed by the MLMC algorithm is
\begin{equation}
	\label{eq:work_mlmc}
	\begin{aligned}
		\WorkMLMC{\vec{\Dx},\vec{M}} & = \sum_{l=0}^LM_l\WorkFVMCFL{\Dx_l}            \\
		                   & =\sum_{l=0}^L\bigO\big(M_l(\Dx_l^{-d-1})\big)          \\
		                   & =\sum_{l=0}^L\bigO\big(M_l2^{l(d+1)}\Dx_0^{-d-1}\big).
	\end{aligned}
\end{equation}
The number of samples per level, $M_l$, has so far been unspecified. It is common to optimize the number of samples for a given convergence rate. We handle the general case, and optimize with respect to the number of samples, where the variance across the levels is abstractly given as
\begin{equation}
	\label{eq:variance_v_l}
	\Var(G(u^{\Dx_l})-G(u^{\Dx_{l-1}}))= V_l.
\end{equation}

\begin{theorem}
Assume 
\[V_l=\bigO(\Dx_l^r)\]
for some positive $r$. Choose the number of samples per level as 
\[M_l=2^{r(L-l)}\Dx_L^{r/2-s},\qquad l=1,\ldots, L,\]
and
\[M_0=\frac{1}{\Dx_L^{2s}}.\]
Then 
\begin{equation}
\label{eq:work_mlmc_optimal}
\WorkMLMC{\vec{\Dx},\vec{M}}=2^{-L(d+1)}\Dx_L^{-d-1-2s}+\Dx_L^{-d-1+r/2-s}L.
\end{equation}
In particular for all $r>0$
\[\WorkMLMC{\vec{\Dx},\vec{M}}\leq \WorkMC{\Dx_L}{\Dx_L^{-2s}},\]
and when $r\geq 2s$, we have
\[\WorkMLMC{\vec{\Dx},\vec{M}}=\WorkFVMCFL{\Dx_L}.\]
\end{theorem}

\begin{proof}
We compute
\begin{align*}
\WorkMLMC{\vec{\Dx},\vec{M}}&=\sum_{l=0}^{L}\Dx_l^{-d-1}M_l\\
&=\sum_{l=0}^{L}(2^{L-l}\Dx_L)^{-d-1}M_l\\
&=(2^{L}\Dx_L)^{-d-1}M_0+\sum_{l=1}^{L}(2^{L-l}\Dx_L)^{-d-1}M_l\\
&=2^{-L(d+1)}\Dx_L^{-d-1-2s}+\sum_{l=1}^{L}(2^{-d-1})^{L-l}\Dx_L^{-d-1}2^{r(L-l)}\Dx_L^{r/2-s}\\
&=2^{-L(d+1)}\Dx_L^{-d-1-2s}+(2^{-d-1+r})^L\Dx_L^{-d-1+r/2-s}\sum_{l=1}^{L}(2^{d+1-r})^{l}\\
&\leq 2^{-L(d+1)}\Dx_L^{-d-1-2s}+(2^{-d-1+r})^L\Dx_L^{-d-1+r/2-s}L(2^{d+1-r})^{L}\\
&= 2^{-L(d+1)}\Dx_L^{-d-1-2s}+\Dx_L^{-d-1+r/2-s}L,
\end{align*}
which is what we wanted.
\end{proof}

\begin{remark}
It is possible to obtain a similar work estimate by minimizing the work for a given error and using a Lagrange multiplier technique \cite{giles}.
\end{remark}

\section{Numerical Examples}\label{sec:numex}
\newcommand{\intvar}{u}
In this section, we will test the Monte Carlo and the multi-level Monte Carlo on a suite of numerical experiments. Our model equation for scalar conservation laws is the one-dimensional Burgers equation
\begin{equation}
	\label{eq:burgers}
	\begin{aligned}
		u_t+\Big(\frac{u^2}{2}\Big)_x & =0      & \qquad \text{for } x\in D:=[0,1],\ t\leq T \\
		u(x,0)                        & =\bar{u}(x) & \qquad\text{for } x\in D
	\end{aligned}
\end{equation}
with suitable boundary conditions. 

We will compute statistical quantities of interest for the approximated solution. These include one-point statistics such as the mean and the variance. We will also compute
a two point \emph{local structure function} of the solution, which is given by
\begin{equation}\label{eq:strucfunclocal}
S_p(x, t; h) := \int_{\R^2}|\intvar_1-\intvar_2|^p\dd \nu^2_{t,x, x+h}(\intvar_1, \intvar_2)\qquad x\in D, h\geq 0
\end{equation}
for some $p\geq 1$. 

We also define the integrated structure functions as
\begin{equation}\label{eq:strucfunc}
\bar{S}_p(t;h)=\int_D S_p(x, t; h)\;dx = \int_{L^1(D)}\int_{D}|u(x+h)-u(x)|^p\;dx\;d\mu_t(u).
\end{equation}
We assume that the initial measure has integrable $L^p$-norms,
\[
\int_{L^1(D)} \|\bar{u}\|_{L^p(\D)}^p\dd \bar{\mu}(\bar{u})<\infty.
\]
By Minkowski's inequality, this assumption implies that the structure function $\bar{S}_p$ is finite, and we may therefore approximate it using the Monte Carlo algorithm. We let $S_p^{\Dx, M}(x,t;h)$, $\bar{S}_p^{\Dx, M}(t;h)$ and $S_p^{\vec{\Dx},\vec{M}}(x,t;h)$, $\bar{S}_p^{\vec{\Dx},\vec{M}}(t;h)$ denote the local and integrated structure functions for the single- and multi-level Monte Carlo methods, respectively. 

We define a perturbed version of the three-point moment by setting

\begin{equation}
\label{eq:3pmom}
M_p(x,t;h_1,h_2) :=\int_{\R^3} (\intvar_1-\intvar_2)(\intvar_1-\intvar_3)^2\dd \nu^3_{t,x,x+h_1,x+h_2}(\intvar_1,\intvar_2,\intvar_3).
\end{equation}
This three-point moment is also bounded assuming e.g.
\[
\int_{L^1(D)} \|\bar{u}\|_{L^3(D)}\dd \bar{\mu}(\bar{u})<\infty.
\]

\subsection{Uncertain shock location}

We consider Burgers' equation with the initial data
\begin{equation}
	\label{eq:shock}
	\bar{u}(\omega, x) = \begin{cases}1 & x<X(\omega) + \frac{1}{2} \\
		0 & \text{otherwise}.\end{cases}\qquad x\in[0,1].
\end{equation}
Here $X$ is a a uniformly distributed random variable on $[-\frac{1}{10},\frac{1}{10}]$. We can explicitly compute the mean and the variance in this case. Given the one-dimensional stochastic space for this scalar problem, the structure function \eqref{eq:strucfunclocal} can also be explicitly computed,
\[S_p(x, t;h) = \begin{cases}0                                                                                              & x+\frac{1}{10}+h<\frac{1}{2}(1+t) \\
		0                                                                                              & x-\frac{1}{10}+h>\frac{1}{2}(1+t) \\
		5\left(\min(\frac{1}{10},x+h-\frac{1}{2}(1+t))-\max(-\frac{1}{10},x-\frac{1}{2}(1+t))) \right) & \text{otherwise.}\end{cases}
\]

We use the Monte Carlo and Multilevel Monte Carlo algorithm to approximate different statistical quantities of interest. For the Monte Carlo simulations we set the number of samples equal to the number of finite volume cells, while for the MLMC simulations we set
\[
M_l=\begin{cases}16\cdot 2^{L-l} & \text{if }l=1,\ldots,L \\
		1/\Dx^L         & \text{if } l=0
	\end{cases}
\]
For the spatial discretization we use a finite volume solver, based on the Godunov flux with piecewise linear WENO reconstruction and an SSP Runge--Kutta method of second order. We set the CFL constant to be $0.475$. 

We plot the point-wise mean and variance in \Cref{fig:mean_var_shock}, and the corresponding convergence histories in \Cref{fig:conv_mean_var_shock}. The plots show that both the MC and MLMC methods approximate the mean and variance quite accurately. The convergence plots also show that the MLMC method is (an order of magnitude) more efficient when compared to the MC method. 

For a comparison between the numerical and analytically computed structure functions, see \Cref{fig:shock_comparison}. In addition, we perform a convergence study with respect to the analytical solution in \Cref{fig:shock2ptconv}. As is clear from \Cref{fig:shock2ptconv}, the work required to achieve a given error in the Monte Carlo simulation is significantly greater than the computational work required with MLMC method. In fact, the gain in efficiency is a couple of orders of magnitude. Moreover, in \Cref{fig:shock2ptconv} (right), we plot the error vs. work for the Monte Carlo and MLMC algorithms in computing the third moment \eqref{eq:3pmom}. Again the gain in efficiency with the MLMC method is significant. 

To confirm the assertion of \Cref{thm:fv},  we measure the Wasserstein distance between the generated sequence and the exact solution in \Cref{fig:wasserstein_shock}. Since the Monte Carlo ensemble is a stochastic quantity, the measured error will also be stochastic. We therefore run $10$ experiments for each mesh resolution and measure the average error. As is clear from the figure, the convergence rate of the mean of the error is close to $1/2$, which agrees well with the Monte Carlo error in \eqref{eq:weak_conv_mc}. However, as we can see, the variance of the error is rather large relative to the error.

Because the initial data \eqref{eq:shock} only depends on a one-dimensional parameter, we can replace the stochastic Monte Carlo integration rule with a deterministic midpoint integration rule. \Cref{fig:wasserstein_shock_midpoint} clearly shows a convergence rate of $1$ for the deterministic integration rule.

\begin{figure}
	\centering
	\begin{subfigure}[b]{0.49\textwidth}
		\InputImage{\textwidth}{0.9\textwidth}{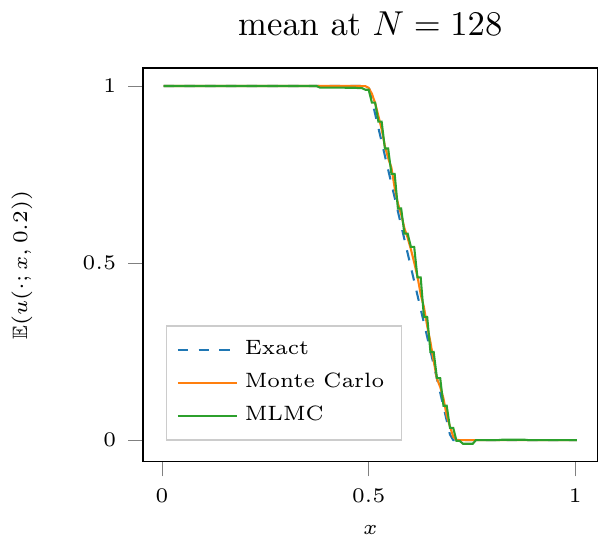}
		\caption{Mean.}
	\end{subfigure}
	\begin{subfigure}[b]{0.49\textwidth}
		\InputImage{\textwidth}{0.9\textwidth}{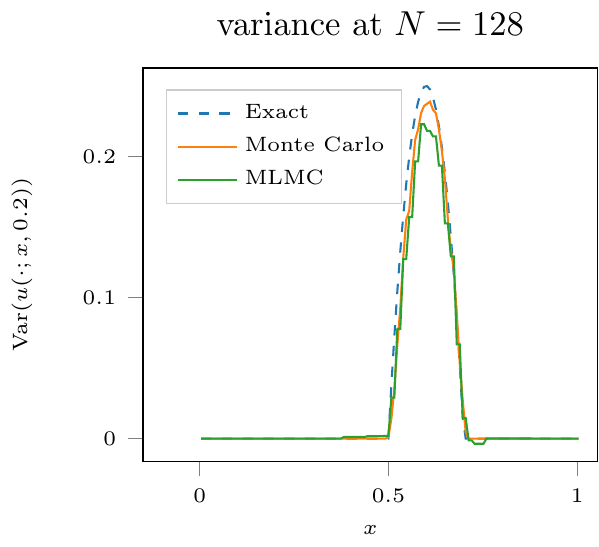}
		\caption{Variance.}
	\end{subfigure}
	\caption{Mean and variance for Burgers' equation and initial data given in \eqref{eq:shock}.\label{fig:mean_var_shock}}
\end{figure}

\begin{figure}
	\centering
	\begin{subfigure}[b]{0.49\textwidth}
		\InputImage{\textwidth}{0.9\textwidth}{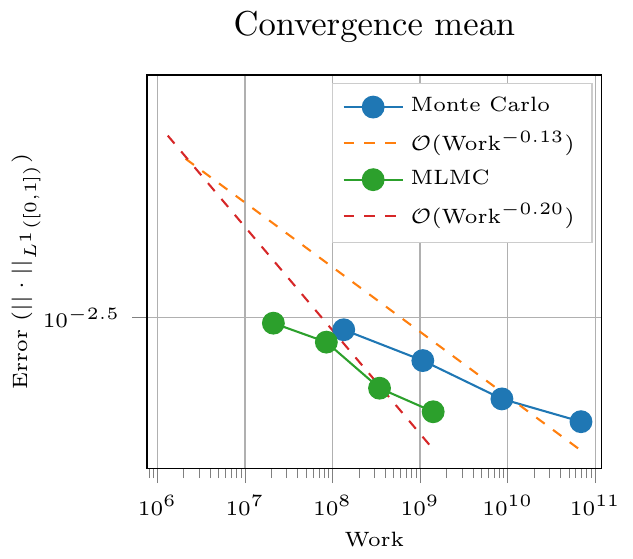}
	\caption{Mean.}
\end{subfigure}
\begin{subfigure}[b]{0.49\textwidth}
	\InputImage{\textwidth}{0.9\textwidth}{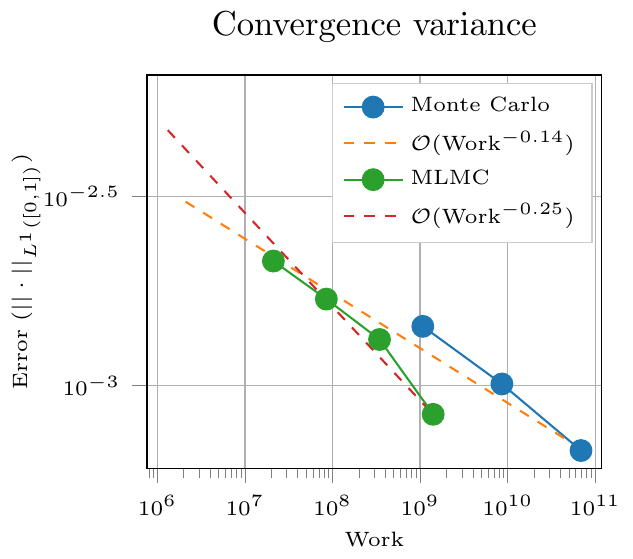}
	\caption{Variance.}
\end{subfigure}
\caption{Convergence for mean and variance for Burgers' equation and initial data given in \eqref{eq:shock}.\label{fig:conv_mean_var_shock}}
\end{figure}

\begin{figure}
	\centering
	\begin{subfigure}[b]{0.49\textwidth}
		\InputImage{\textwidth}{0.9\textwidth}{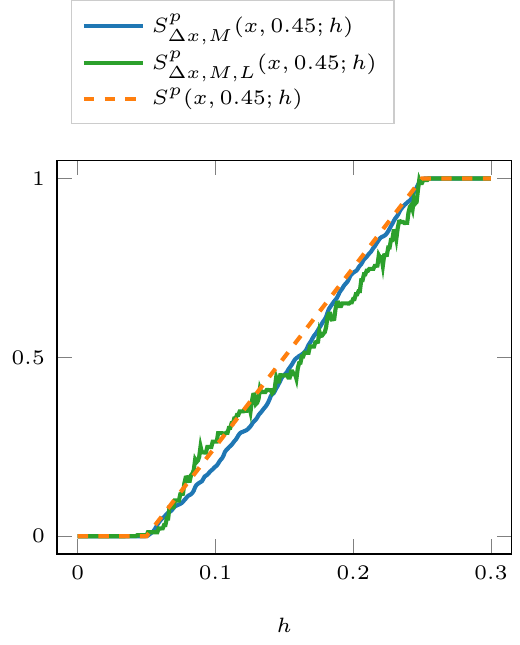}
		\caption{Point quantity.}
	\end{subfigure}
	\begin{subfigure}[b]{0.49\textwidth}
		\InputImage{\textwidth}{0.9\textwidth}{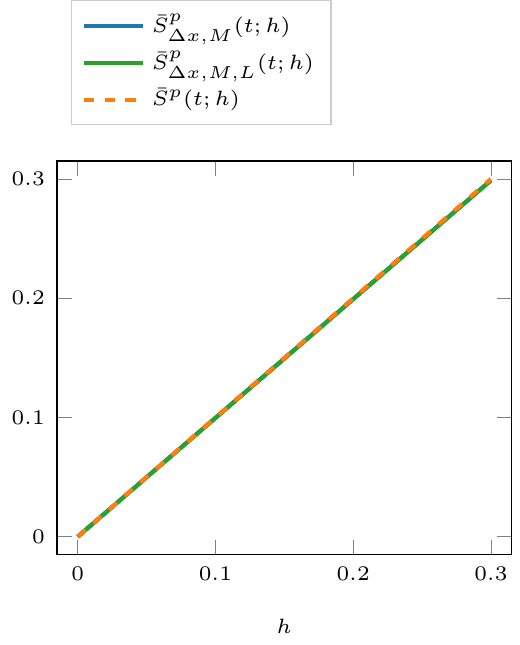}
		\caption{Integrated.}
	\end{subfigure}
	\caption{Comparison of two point structure function numerically computed and analytic solution for Burgers' equation and initial data given in \eqref{eq:shock}. Here $\Dx=1/1024$ and $M=1024$.}
	\label{fig:shock_comparison}
\end{figure}


\begin{figure}
	\centering
	\begin{subfigure}[b]{0.49\textwidth}
		\InputImage{\textwidth}{0.8\textwidth}{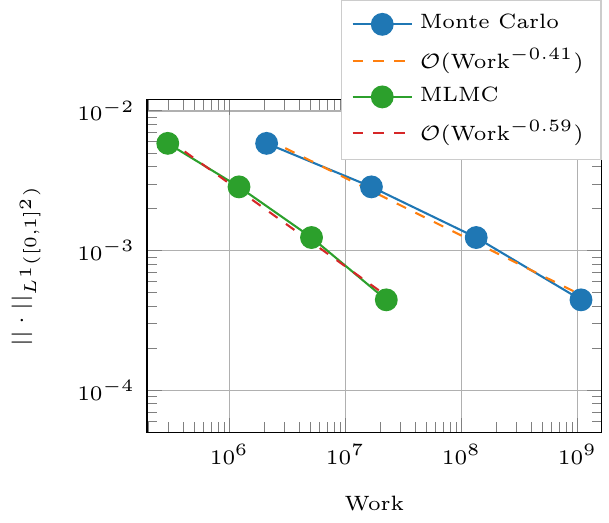}
		\caption{Two-point structure functions.}
	\end{subfigure}
	\begin{subfigure}[b]{0.49\textwidth}
		\InputImage{\textwidth}{0.8\textwidth}{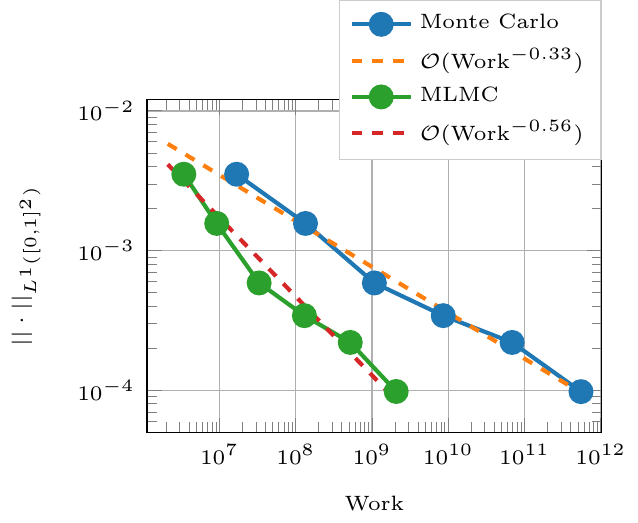}
		\caption{Three-point moment.}
	\end{subfigure}
	\caption{Convergence errors of the two-point structure function \eqref{eq:strucfunc} and three-point moment \eqref{eq:3pmom} with initial data given as the uncertain shock location \eqref{eq:shock}.}
	\label{fig:shock2ptconv}
\end{figure}

\begin{figure}
	\begin{subfigure}[b]{0.49\textwidth}
		\InputImage{\textwidth}{0.9\textwidth}{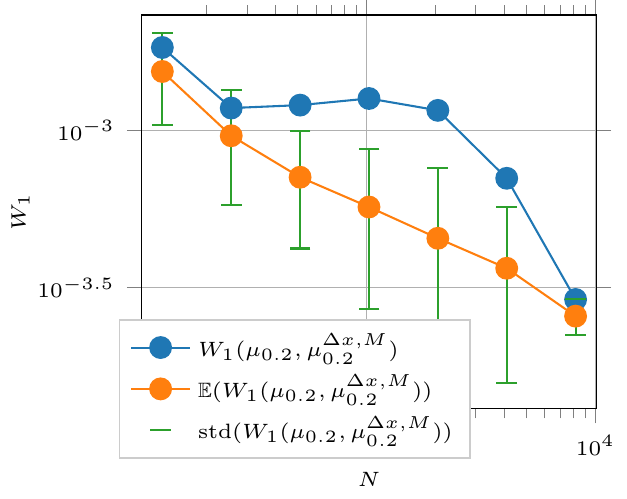}
		\caption{Monte Carlo algorithm}
		\label{fig:wasserstein_shock}
	\end{subfigure}
	\begin{subfigure}[b]{0.49\textwidth}
		\InputImage{\textwidth}{0.9\textwidth}{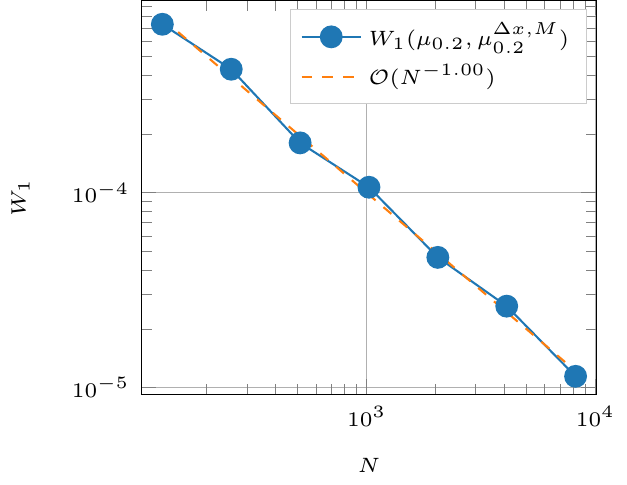}
		\caption{Midpoint rule}
		\label{fig:wasserstein_shock_midpoint}
	\end{subfigure}
	\caption{Wasserstein convergence for the initial data given in \eqref{eq:shock} at $T=0.2$.}
\end{figure}

\subsection{Fractional Brownian motion}
As is standard in the \emph{burgulence} literature \cite{fris1,fris2}, we perform experiments for Burgers' equation \eqref{eq:burgers} with the initial data set to be \emph{fractional Brownian motion}. Introduced by Mandelbrot et al.~\cite{mandelbrot_fractional}, fractional Brownian motion can be seen as a generalization of standard Brownian motion with a scaling exponent different than $1/2$. We set
\[
u^H_0(\omega; x):=B^H(\omega; x)\qquad \omega\in\Omega,\ x\in [0,1],
\]
where $B^H$ is fractional Brownian motion with Hurst exponent $H\in(0,1)$. Brownian motion corresponds to a Hurst exponent of $H=1/2$. 

To generate fractional Brownian motion, we use the random midpoint displacement method originally introduced by L\'evy~\cite{Levy1965} for Brownian motion, and later adapted for fractional Brownian motion~\cite{Fournier:1982:CRS:358523.358553,Voss1991}. Consider a uniform partition $0=x_\hf<\dots<x_{N+\hf}=1$ with $x_\iphf - x_\imhf \equiv \Dx$, where $N=2^k+1$ is the number of cells for some $k\in\N$. We first fix the endpoints
\[u^{H, \Delta x}_{1}(\omega; 0) =0\qquad u^{H, \Delta x}_{N}(\omega;0)=X_0(\omega), \]
where $(X_k)_{k\in\N}$ is a collection of normally distributed random variables with mean 0 and variance 1. Recursively, we set
\[u^{H, \Delta x}_{2^{k-l-1}(2j+1)}(\omega; 0)=\frac{1}{2}\left (u^{H, \Delta x}_{2^{k-l}(j+1)}(\omega; 0)+u^{H, \Delta x}_{2^{k-l}j}(\omega; 0)\right)+\sqrt{\frac{1-2^{2H-2}}{2^{2lH}}}X_{2^l+j}(\omega)\]
for $l=0,\ldots,k$ and for $j=0,\ldots,2^l$. That is, we bisect every interval and set the middle value to the average of the neighbouring values plus some Gaussian random variable. See \Cref{fig:fractional_sample} (left) for an example with $H=0.01$.

We run experiments for the standard Brownian motion ($H=0.5$) and an even rougher (pathwise) initial datum, corresponding to $H=0.01$ . In both cases, one can easily check that the initial total variation ${\rm TV}(\bar{u}^H(\omega,\cdot))$ is in fact infinite for almost all $\omega$, and hence the Kuznetszov error estimate \cite{KUZNETSOV1976105} does not apply. 

For the Monte Carlo simulations we set the number of samples equal to the number of finite volume cells, while for the MLMC simulations we set
\[M_l=\begin{cases}16\cdot 2^{L-l} & \text{if }l=1,\ldots,L \\
		1/\Dx^L         & \text{if } l=0
	\end{cases}
\]
For the Finite Volume Method, we use the Godunov flux with WENO2 reconstruction and an SSP Runge--Kutta method of second order. We set the CFL constant to be $0.475$.

We measure the decay of $V_l$ defined in \eqref{eq:variance_v_l} numerically for $H=0.01$ and $H=0.5$ and see that $V_l$ as a function of the mesh width $h$ behaves approximately as $\bigO(h)$ (results not displayed here). This happens in spite of the fact that the Kuznetszov error estimate does not apply.

Our aim is to numerically approximate structure functions. Following \cite{fris2}, one intuitively argues that if the (pathwise) solution consists of a disjoint set of shocks, well-separated by rarefaction waves for almost all times, then the exponent of the structure function \eqref{eq:strucfunc} will be dominated by the behavior at shocks and the structure functions will scale as $\mathcal{O}(h)$ for any $1 \leq p < \infty$. Given the rigorous results of Sinai \cite{sinai}, one expects this scaling to hold in the case of standard Brownian motion initial data. 

In \Cref{fig:fractional_scaling} we approximate the scaling exponents of $\bar{S}_p$ for different values of $H$. For $H=0.5$,  the structure functions scale close to $\mathcal{O}(h)$ for $p=1,2,3$, which agrees with the results of \cite{fris1,sinai}. Similarly for $H=0.01$, even though the initial data is much rougher than standard Brownian motion, we observe from \Cref{fig:fractional_sample}, that the initally highly oscillatory pathwise solution very quickly evolves into a set of well-separated shocks separated by rarefaction waves. As argued before, one expects a scaling exponent of approximately $\mathcal{O}(h)$ for all $p$. This is indeed verified in the approximate structure functions for $p=1,2,3$, shown in \Cref{fig:fractional_scaling} (right). 

We also measure the numerical convergence rate of the Monte Carlo and Multilevel Monte Carlo algorithm against a reference solution computed with $\Delta x = 1/4096$ and $M=4096$. The results are shown in \Cref{fig:fractional_0.01_convergence}. As is expected, the MLMC algorithm outperforms the single-level Monte Carlo algorithm. Indeed, the error versus work asymptotic of a MLMC simulation is equal to that of the runtime asymptotics for a single sample.

This example clearly illustrates the ability of both the MC and MLMC methods to compute statistical quantities in a realistic and well-studied problems and demonstrates the considerable gain in efficiency for the MLMC method over the MC method.

\subsection{Cubic conservation law with Brownian initial data}
We redo the experiments in the previous section, but with the cubic conservation law,
\begin{equation}
\label{eq:cubic}
\begin{aligned}
u_t+\Big(\frac{u^3}{3}\Big)_x & =0      & \qquad \text{for } x\in D:=[0,1],\ t\leq T \\
u(x,0)                        & =\bar{u}(x) & \qquad\text{for } x\in D.
\end{aligned}
\end{equation}
In this case, it is not possible to derive explicit expressions or asymptotices for the structure functions as in the case of Burgers' equation \cite{sinai}, as no explicit Hopf--Lax-type formulas are available for a non-convex flux function. Therefore, numerical simulations are the main tool in calculating statistical quantities of interest.  We approximate the structure functions \eqref{eq:cubic} with initial data given as Brownian motion, in other words, we set $u_0(x,\omega)=B^{1/2}_x(\omega)$. The results are shown in \Cref{fig:cubic_scaling}. As is clear from the figure, we get the same linear scaling of the structure functions for any $1 \leq p < \infty$, indicating that the initially highly oscillatory solution breaks down into shocks, well separated by rarefactions, as in the case of Burgers' equation.

\begin{figure}

\InputImage{\textwidth}{0.7\textwidth}{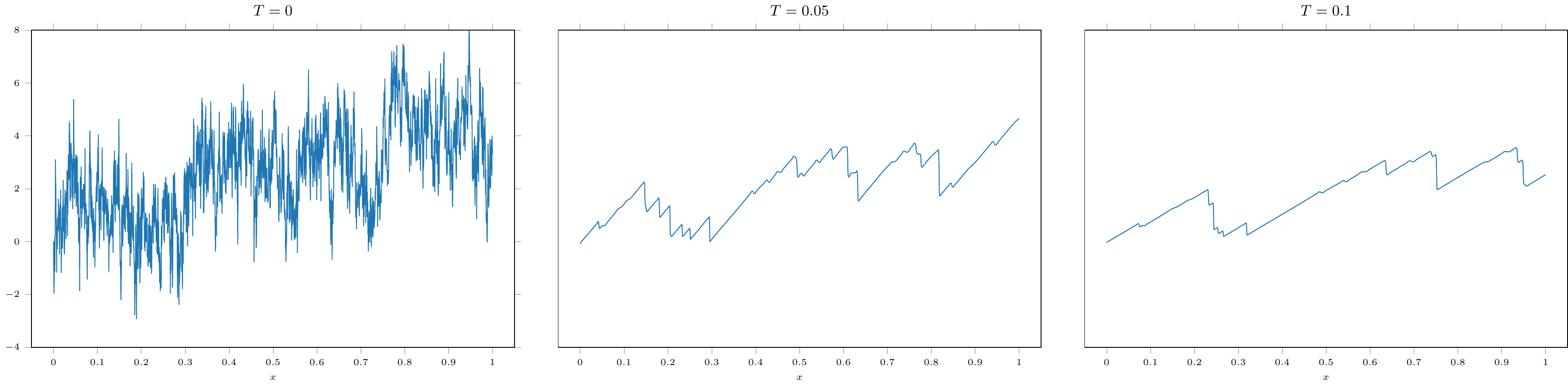}

	\caption{Single sample simulation results for the Burgers' equation with initial data given as fractional Brownian motion with $H=0.01$. In this example, $\Dx=1/1024$.}
	\label{fig:fractional_sample}
\end{figure}

\begin{figure}
	\begin{subfigure}[b]{0.49\textwidth}
		\InputImage{\textwidth}{0.8\textwidth}{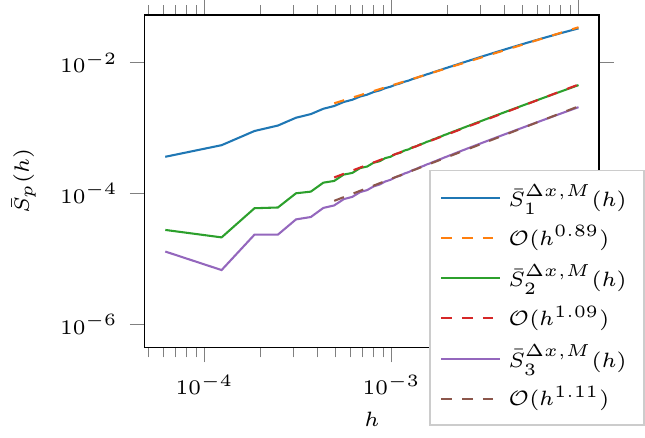}
		\caption{Fractional Brownian motion with $H=0.5$}
		\label{fig:fractional_0.5_scaling}
	\end{subfigure}
	\begin{subfigure}[b]{0.49\textwidth}
		\InputImage{\textwidth}{0.8\textwidth}{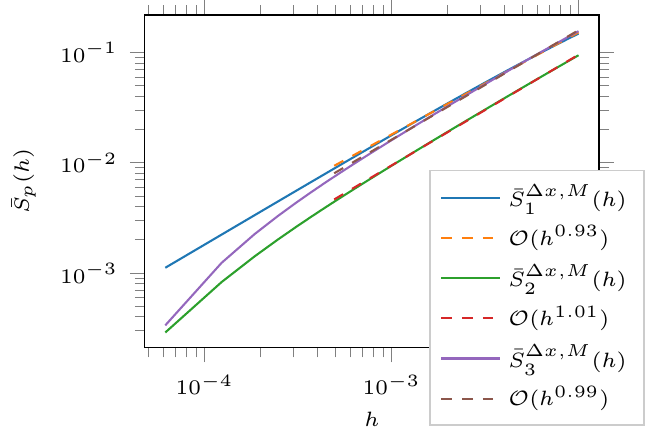}
		\caption{Fractional Brownian motion with $H=0.01$}
		\label{fig:fractional_0.01_scaling}
	\end{subfigure}	
	\caption{Structure functions for Burgers' equation with initial data given as fractional Brownian motion. In this example, $\Dx=1/16384$ and $M=16384$.}
	\label{fig:fractional_scaling}
\end{figure}

%

\begin{figure}
	\begin{subfigure}{0.49\textwidth}
			\centering
	\InputImage{\textwidth}{0.75\textwidth}{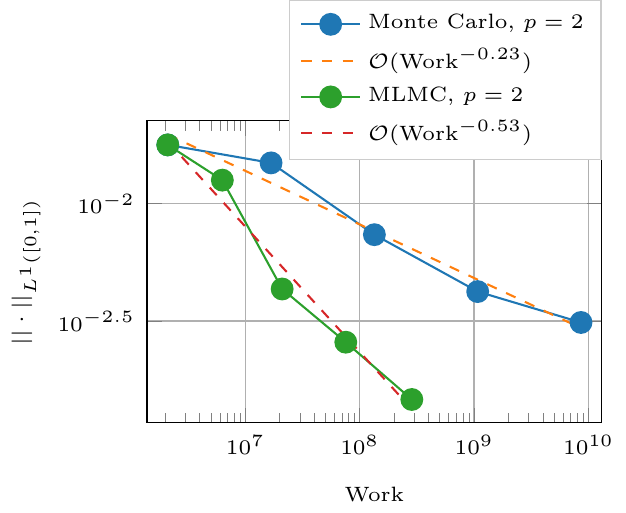}
		\caption{Numerical error for the two point structure function against a reference solution for Burgers' equation with Brownian motion as initial data.}
	\label{fig:brownian_convergence}
\end{subfigure}
\begin{subfigure}{0.49\textwidth}
		\centering
	\InputImage{\textwidth}{0.75\textwidth}{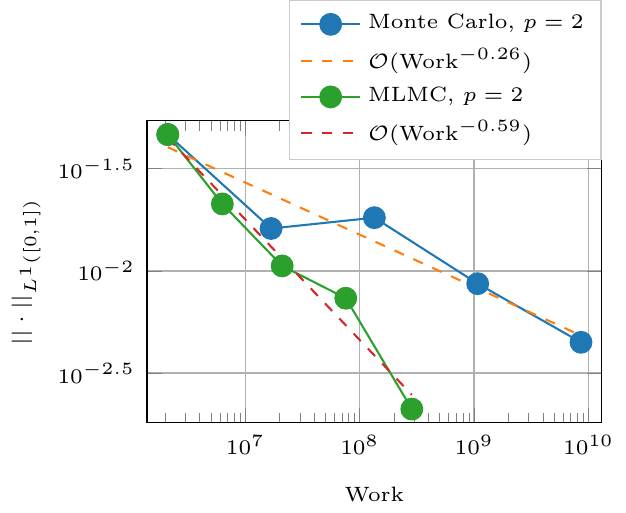}
	\caption{Numerical error for the two point structure function against a reference solution for Burgers' equation with fractional Brownian motion with $H=0.01$ as initial data.}
	\label{fig:fractional_0.01_convergence}
\end{subfigure}
\caption{Error vs.~work for the Monte Carlo and MLMC methods for Burgers' equation with fractional Brownian motion initial data.}
\end{figure}


%


\begin{figure}
	\centering
	\InputImage{\textwidth}{0.75\textwidth}{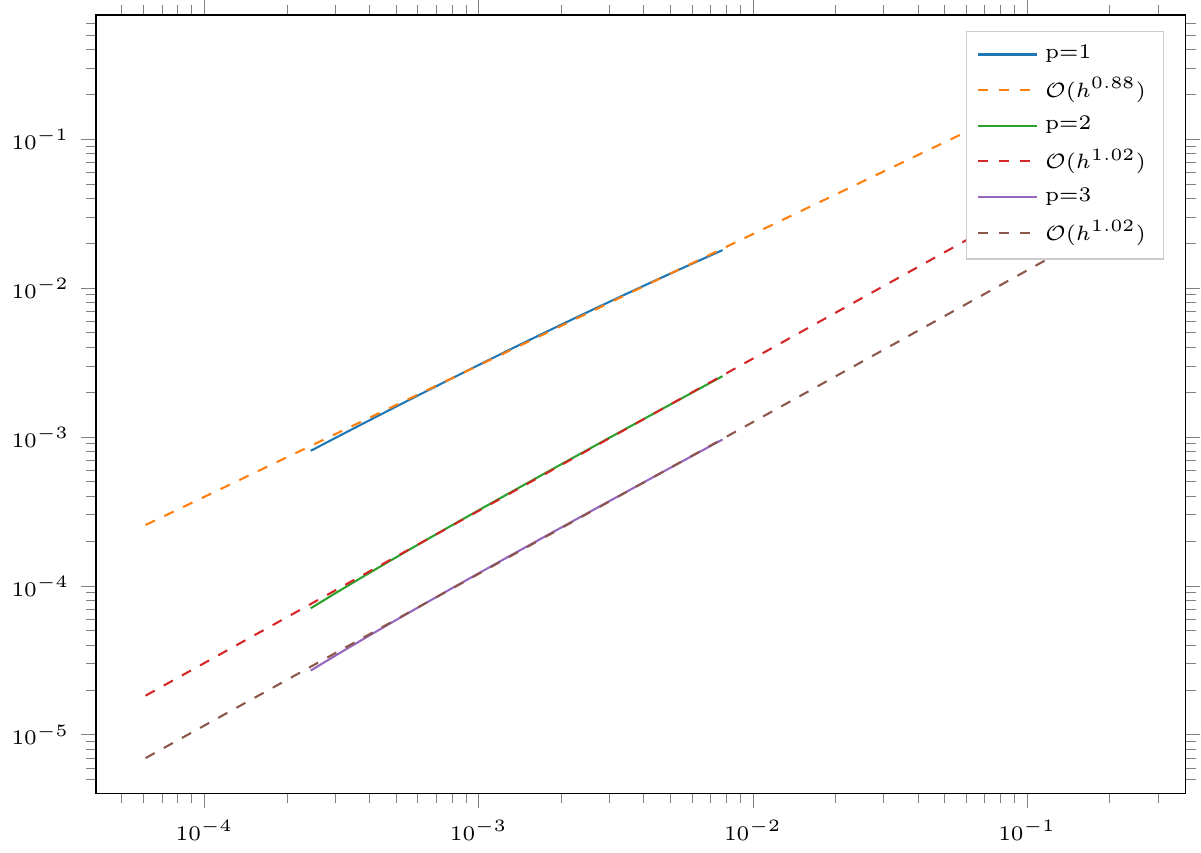}
	\caption{Structure function scaling for a cubic conservation law ($f(u)=1/3u^3$).}
	\label{fig:cubic_scaling}
\end{figure}


\section{Conclusion}
\label{sec:conc}
We consider statistical solutions for scalar conservation laws. Statistical solutions are time-parametrized probability measures on $L^1(\R^d)$, whose time evolution is specified in terms of an infinite family of PDEs for the corresponding correlation measures. Statistical solutions provide a framework for \emph{parametrization free} uncertainty quantification in conservation laws with random initial data and their well-posedness was established in \cite{FLM17}.  

We design an efficient algorithm for the numerical approximation of statistical solutions of scalar conservation laws. This algorithm is based on standard finite volume (difference) methods for the spatio-temporal discretization and (multi-level) Monte Carlo methods for discretizing the probability space. 

We prove that both the Monte Carlo and the multi-level Monte Carlo based algorithms converge to the \emph{entropy statistical solutions} of scalar conservation laws, in the Wasserstein metric on probability measures on integrable functions, as the mesh is refined and the number of samples increased. We also present a complexity analysis and prove that there is a considerable gain in computational efficiency of the multi-level Monte Carlo method over the Monte Carlo method.

We present a set of numerical experiments for Burgers' equation to illustrate the ability of the both the Monte Carlo and the multi-level Monte Carlo algorithms to approximate the statistical solution accurately and to demonstrate the gain in efficiency resulting from the MLMC method. These experiments involve rough random initial conditions such as fractional Brownian motion (in space), and the computed results are consistent with those published in the burgulence literature \cite{fris1,fris2}. We also compute statistical quantities of interest for the cubic conservation law. In this case, it is not possible to obtain analytical formulas for structure functions. However, numerical results show that the structure functions scale as in the case of Burgers' equation with rough random initial conditions. 

The convergence of our numerical methods in the scalar case was underpinned by the presence of (pathwise) convergent numerical methods for the underlying deterministic problem. Such convergence results are not available for systems of conservation laws. Consequently, the design of convergent numerical approximations for systems of conservation laws is very challenging and is addressed in a forthcoming paper.

\section*{Acknowledgements}
USF was supported in part by the grant \textit{Waves and Nonlinear Phenomena} (WaNP) from the Research Council of Norway. Parts of this research was conducted while USF was visiting the Seminar for Applied Mathematics and he would like to thank the hosts for their warm hospitality.

SM was partially supported by ERC STG NN. 306279, SPARCCLE.


\bibliography{biblo}{}
\bibliographystyle{plain}

\end{document}